\documentclass[fleqn,a4paper,12pt]{article}
\usepackage{amsmath,amsfonts,calrsfs,amssymb,color}
\usepackage{amsthm,dsfont}

\newtheorem{Theorem}{Theorem} 

\newtheorem{Proposition}[Theorem]{Proposition}
\newtheorem{Lemma}[Theorem]{Lemma}
\newtheorem{Corollary}[Theorem]{Corollary}
\newtheorem{Remark}[Theorem]{Remark}
\newtheorem{Example}[Theorem]{Example}
\newtheorem{Hypothesis}{Hypothesis}

\def\R{\mathbb R}
\def\N{\mathbb N}

\def\E{\mathbb E}
\def\P{\mathbb P}
\def\Q{\mathbb Q}
\def\ds{\displaystyle}

\newcommand{\one}{1\!\!\!\;\mathrm{l}}

\title{\bf  Cauchy--Dirichlet problems for  a class of hypoelliptic equation in $\R^d$: a new     probabilistic representation formula for the gradient of the solutions}

\newlength{\tempa}
\divide\tempa by 2
\newlength{\tempb}
 
\author{Giuseppe Da Prato\\
\small Scuola Normale Superiore, Pisa, Italy
\and  
Luciano Tubaro\\
\small University of Trento, Trento, Italy
}

\begin{document}
\maketitle
\begin{abstract}  
We are concerned with     an Ornstein--Uhlenbeck process   $X(t,x)=e^{tA}x +\int_0^t e^{(t-s)A}\sqrt C\,dW(s)$ in $\R^d$,  $d\ge 1$, where $A$ and $C$ are $d\times d$ matrices,   $C$ being
 semidefinite positive.  Our basic assumption is that the matrix $Q_t=\int_0^te^{sA}Ce^{sA^*}ds$ is non singular for all $t>0$;  this implies that  the corresponding Kolmogorov operator is hypoellyptic. Then we consider  the stopped semigroup $ R^{\mathcal {O}_r}_T\varphi(x)=\E\left[\varphi(X(T,x))\one_{T\le\tau^r_x}\right],\; T\ge 0$ where $\mathcal O_r=\{g<r\}$ is bounded, $g$ is convex, and
  $\tau^r_x= \inf\{t>0:\;X(t,x)\in \overline{\mathcal O_r}^{\,c}\}$.
  We prove the existence  and a new representation formula for the gradient   of $R^{\mathcal {O}_r}_T\varphi$,   where $T>0$ and  $\varphi$ is    bounded and Borel.
  
 
\end{abstract}
\bigskip

\noindent {\bf 2000 Mathematics Subject Classification AMS}:  35J15,  60G53, 60H99, 60J65.

 \bigskip

\tableofcontents

 \section{Introduction and setting of the problem}
 
 We are  here concerned with  an Ornstein--Uhlenbeck process in   $\R^d=H,\; d\in\N$,
 \begin{equation}
 \label{e1}
 X(t,x)=e^{tA}x+\int_0^te^{(t-s)A}\sqrt C\;dW(s),\quad x\in H,\; t\ge 0,
 \end{equation}
 where  $A$ and $C$ are $d\times d$ matrices, $C$  being symmetric and semi-definite positive. Moreover $W(t),\;t\ge 0$, represents an $H$--valued standard   Wiener process   defined on a   probability space $(\Omega, \mathcal F, \P)$. We denote by  
$W_A$   the {\em stochastic convolution}
 \begin{equation}
 \label{e2}
  W_A(t):=\int_0^te^{(t-s)A}\sqrt C\;dW(s),\quad  t\ge 0.
 \end{equation}
 Our   basic assumption is the following.
 \begin{Hypothesis}
 \label{h1}
 The matrix  $ Q_t:=\int_0^te^{sA}Ce^{sA^*}ds$ is non singular for all  $t>0.$  
   \end{Hypothesis}
   \begin{Remark}
   \em Hypothesis 1 arises in controllability problems
   for the deterministic    system $D_t\xi=A\xi +\sqrt C\,u$, where
   $u$ is a control. See e.g. \cite{Za92}.

   \end{Remark}

As well known, the transition semigroup $R_t,\,t\ge 0,$ corresponding to the process $X(t,x)$ is given by 
 \begin{equation}
 \label{e5a}
 R_t\varphi(x)=\int_H\varphi(y) N_{e^{tA}x,Q_t}(dy) ,\quad  t\ge 0,\;x\in H,\;\varphi\in B_b(H),
 \end{equation}
 where  
 $ N_{e^{tA}x,Q_t}$ is the gaussian  probability measure on $H$ of mean $e^{tA}x$ and covariance  $Q_t$ and  $B_b(H)$   denotes the space of all mappings $H\to H$ which  are bounded and Borel.
By Hypothesis 1  the matrix $\Lambda_t:=Q_t^{-1/2}e^{tA}$ is  non singular for all $t>0$; consequently, by the Cameron--Martin Theorem it follows that 
 $ N_{e^{tA}x,Q_t}<<  N_{Q_t}$  and
\begin{equation}
 \label{e7}
\frac{d N_{e^{tA}x,Q_t}}{d N_{Q_t}}(y)=e^{-\frac12| \Lambda_t x|_H^2+\langle \Lambda_t x,Q_t^{-1/2}y  \rangle_H},\quad t>0,\, y\in H.
 \end{equation}
 Therefore      a  well known representation formula for $R_t$  follows by \eqref{e5a},
  \begin{equation}
 \label{e8}
 R_t\varphi(x)=\int_H\varphi(y)e^{-\frac12| \Lambda_t x|_H^2+\langle \Lambda_t x,Q_t^{-1/2}y  \rangle_H} N_{Q_t}(dy),\quad  t\ge 0,\;x\in H,\;\varphi\in B_b(H).
 \end{equation}
By \eqref{e8} we can deduce that  $R_t\varphi$ is differentiable infinitely many times, in particular it is strongly Feller.\bigskip 

The goal of this paper is to generalise the above regularity results
to the stopped semigroup
 $R^{\mathcal {O}_r}_T,\;T\ge 0,$  defined by 
 \begin{equation}
 \label{e10}
\ds R^{\mathcal {O}_r}_T\varphi(x)=\E\left[\varphi(X(T,x))\one_{T\le\tau_x}\right],\quad T\ge 0,\; 
\varphi\in B_b(\overline{\mathcal O_r}),
 \end{equation}
 where $\mathcal {O}_r$ is an open convex  bounded subset of $H$
and  $\tau_x$ is the exit time from  $\overline{\mathcal O_r}$.\medskip

 More precisely, we shall assume    
       \begin{Hypothesis}
  \label{h2}
  (i)  $g:H\to \R$ is a convex function  of class $C^1$ such that $g(0)=0$,  $g(x)>0$ and $g'(x)\neq 0$ for all $x\neq 0$. For any $r>0$ we set
  $\mathcal O_r=\{g<r\}$, $\overline{\mathcal O_r}=\{g\le r\}$ and  $\partial\mathcal O_r=\{g^{-1}(r)\}$.  Moreover,    $\mathcal O_r$ is bounded. \medskip
  
  \noindent (ii) There exist  $a,\,b>0$ such that $|g(x)|
  +|g'(x)|_H\le a+e^{b|x|_{H}}$for all $x\in H$.\medskip

   \end{Hypothesis}
      The semigroup $R^{\mathcal {O}_r}_T,\;T\ge 0,$  is related, as well known,  to  the Dirichlet problem in $\overline{\mathcal O_r}$ for the Kolmogorov operator,
 \begin{equation}
\label{e10d}
 \mathcal K \varphi:=\frac12\,\mbox{\rm Tr}\,[CD^2\varphi]+\langle Ax,D\varphi\rangle.
\end{equation} 
This problem is elliptic when the matrix $C$ is non singular, otherwise  is hypoelliptic. In the last case the existence of the gradient of  $R^{\mathcal {O}_r}_T\varphi,\;T\ge 0,$  when $\varphi$ is namely bounded and Borel, is more challenging.

 Here  is a simple example where Hypothesis \ref{h1} is fulfilled.
\begin{Example}
\label{ex1}
\em 
 Let $d=2$ and   
$
A=\left(
\begin{array}{cc}
0&0\\
1&0
\end{array} \right),\;
C=\left(
\begin{array}{cc}
1&0\\
0&0
\end{array} \right).
$
Then we have 
\begin{displaymath}
Q_t=\int_0^te^{sA}Ce^{sA^*}ds=\int_0^t\left(
\begin{array}{cc}
1&s\\
s&s^2
\end{array} \right)ds=
\left(\begin{array}{cc}
t&t^2/2\\
t^2/2&t^3/3
\end{array}\right).
\end{displaymath}
Therefore det $Q_t>0$ for any $t>0$. If $f\in B_b(H)$ and $t>0$, we   conclude that $R_tf$ is of class $C^\infty$. Note that $u(t,\xi)=R_tf(\xi),$   is the solution  of the well known Kolmogorov equation
\begin{equation}
\label{e9}
\left\{\begin{array}{l}
\ds D_tu(t,\xi_1,\xi_2)=\tfrac{1}{2}\,D^2_{\xi_1}u(t,\xi_1,\xi_2)+\xi_1D_{\xi_2}u(t,\xi_1,\xi_2) \\
\\
u(0,\xi)=f(\xi),\quad \xi=(\xi_1,\xi_2)\in \R^2,
\end{array}\right.
\end{equation}
which is {\em hypoelliptic}. 
\end{Example}

\medskip

 Let us explain our result. We start  from  an obvious consequence of  \eqref{e10},
    \begin{equation}
 \label{e9h}
 R^{\mathcal {O}_r}_T\varphi(x) =\int_{\{g(e^{sA}x+h(s))\le r,\,\forall\;s\in[0,T]\}}\varphi(h(T)+e^{TA}x) N_{\mathbb Q_T}(dh),\quad \varphi\in B_b(\overline{\mathcal O_r}),
 \end{equation}
 where   $ N_{\Q_T}$ is the law of $W_A(\cdot)$  in $X:=L^2(0,T;H)$ or in $E:=C([0,T];H)$,  see Lemma \ref{l1} below.

  Let $x\in H$; then we   cannot eliminate $x $  in identity \eqref{e9h}  making the translation $h\to h-e^{\cdot A}x$
 and  using the Cameron--Martin formula
because  the  measures  $N_{ e^{\cdot A}x,\Q}$ and $N_{\Q}$  
are singular. 
For this reason we look for another translation   $h\to h-a(x,\cdot)$   such that $a(x,\cdot)$ belongs to  $\Q_T(X)$ for all $x\in H$   (and a--fortiori to 
 $\Q^{1/2}_T(X)$,  the Cameron--Martin space of $ N_{\mathbb Q_T}$)  and {\em such that}:  
  \begin{equation}
 \label{e10h}
 a(x,T)=e^{TA}x,\quad \forall\,x\in H
 \end{equation}
 (see Proposition \ref{p5c} below).
Then the measures  $N_{a(x,\cdot),\mathbb Q_T}$ and $N_{\mathbb Q_T}$  are equivalent, so that   by the Cameron--Martin Theorem we have
\begin{equation}
 \label{e11h}
 \frac{dN_{a(x,\cdot),\mathbb Q_T}}{dN_{\mathbb Q_T}}(h)=\exp\left\{-\tfrac12|\Q_T^{-1/2}a(x,\cdot)|_X^2+W_{\Q_T^{-1/2}a(x,\cdot)}(h)\right\},\quad x\in H,\;h\in X,$$
  \end{equation} 
  where $\Q_T^{-1/2}$ is the pseudo inverse of $\Q_T^{1/2}$, see e.g. \cite[Theorem 2.23]{DaZa14}.
  Now we take advantage of the special form of $a(x,\cdot)$ for simplifying identity  \eqref{e9h}. We write
  \begin{equation}
 \label{e12h}
  \ds |\Q_T^{-1/2}a(x,\cdot)|_X^2=\langle \Q_T^{-1}a(x,\cdot) ,a(x,\cdot)  \rangle_X=:F(x),
  \end{equation} 
  and,
 \begin{equation}
 \label{e13h}
  W_{\Q_T^{-1/2}a(x,\cdot)}(h)=\langle \Q_T^{-1/2}a(x,\cdot) ,\Q_T^{-1/2}h  \rangle_X=
  \langle \Q_T^{-1}a(x,\cdot) ,h  \rangle_X=
 :G(x,h).
   \end{equation} 
Note that  both $F$ and $G$ are regular.
Now \eqref{e11h} becomes  
   \begin{equation}
 \label{e14h}
 \frac{dN_{a(x,\cdot),\mathbb Q_T}}{dN_{\mathbb Q_T}}(h)=\exp\left\{-\tfrac12\,F(x)+G(x,h)\right\},\quad x\in H,\;h\in X,
 \end{equation} 
 and  \eqref{e9h}   
  \begin{equation}
 \label{e2.8}
R^{\mathcal O_r}_T\varphi(x)=\int_{\{\Gamma(h+ d(x,\cdot) \le r\}} \varphi(h(T)) \exp\left\{-\tfrac12\,F(x)+G(x,h)\right\}N_{\mathbb Q_T}(dh),\quad \varphi\in B(\overline{\mathcal O_r}),
 \end{equation}
 where
  \begin{equation}
 \label{e16h}
 \Gamma(h+d(x,\cdot))=\sup_{t\in[0,T]}\,g(h(t)+d(x,t)),\quad k\in E,\quad d(x,t)=e^{tA}x-a(x,t),\quad t\in[0,T].
 \end{equation}
 In  the integral \eqref{e2.8} the variable $x$ does not appear under  the argument of $\varphi$.  Since the mapping $x\to
 \Gamma(h+d(x,\cdot))$ is continuous this fact   implies that the semigroup $R^{\mathcal O_r}_T,\,T>0$, is {\em strong Feller}   that is $\varphi\in B_b(H)\Rightarrow R^{\mathcal O_r}_T\varphi\in C_b(H)$ for all $T>0$, see Proposition \ref{p1} below. 
 
   More difficult is  to show that   $R^{\mathcal O_r}_T\varphi$ is differentiable  for all $T>0$. As it is expected,  this will produce a surface integral which, unfortunately,  is not covered by the classical assumptions from Airault--Malliavin, \cite{AiMa88} see also \cite{DaLuTu14}.  To overcome this difficulty  we introduce in Section 3 an approximating semigroup $R^{\mathcal O_r}_{T,n},\,T>0$, for all  decomposition $\{t_j=\tfrac{jT}{2^n},\,j=0,1,\ldots,2^n\}$  of $[0,T]$,   namely   by approximating any function $h$ from $E$ by step functions.
Then we arrive to an identity for   $R^{\mathcal O_r}_{T,n}\varphi(x)$ (see \eqref{e35a})  that  can be  easily  differentiated with respect to $x$, see identity \eqref{e38a}.
 It remains to let  $n\to\infty$;     this is not  easy  due to the factor
 $$
 \langle \Q_{T,n}^{-1/2} (d_x(x,\cdot )y), \Q_{T,n}^{-1/2}\,h\rangle_{H^{2^n}}$$ which appears in the identity \eqref{e38a} because $d_x(x,\cdot )y$  does not belong
 to the Cameron--Martin space of  $N_{ \Q_{T}}$.
   Some  additional  work is required, based on the Ehrhard inequality for the gaussian measure $N_{\Q_T}$ and   the selection principle of Helly. After some  manipulations,   we
   arrive at the representation formula \eqref{e100} which is the main result of the paper. 
   Our procedure was  partially inspired by  a  paper by Linde \cite{Li86},  which was dealing, however, with a completely different situation.
   
We believe that our method could be extended to   more general Kolmogorov operators of the form
 \begin{equation}
\label{e10dd}
 \mathcal K_1 \varphi=\frac12\,\mbox{\rm Tr}\,[CD^2\varphi]+\langle Ax+b(x),D\varphi\rangle,\quad \varphi\in C^2(\overline{\mathcal O_r}).
\end{equation}  
where $b:H\to H$ is  suitable nonlinear mapping. This will be the object of a future work.


\bigskip

We end this section with some notation.
  For any $T>0$ we consider  the law of $X(\cdot ,x)$ both  in the Banach space $E=C([0,T];H)$ and in the Hilbert space $X=L^2(0,T;H)$ (in the second case it is concentrated on $E$ which is a Borel subset of $X$). 
We shall denote by $|\cdot|_X$ (resp. $|\cdot|_E$)   the norm of $X$ (resp. of $E$). The scalar product  from two elements $x,y\in H$ (resp. $X$)  will be denoted 
   either by $\langle x, y \rangle_H$ (resp. $\langle x, y \rangle_X$)  or  by $x\cdot y$. If $\varphi\in C^1_b(E)$ and $\eta\in E$ we denote by $D\varphi(h)\cdot \eta$ the derivative of $\varphi$ at $h$ in the direction $\eta$.

   In what follows several integrals with respect to $dN_\Q$ will be considered, according to the convenience, both in $X$ and in $E$.  

\section{Strong Feller property of $R^{\mathcal O_r}_{T},\,T>0$}

We first recall some properties of the gaussian measure $\Q_T$.
The following lemma is well known, see e.g \cite[Theorem 5.2]{DaZa14}.
\begin{Lemma}
\label{l1}
The  law of $W_A(\cdot)$   is gaussian $ N_{\mathbb Q_T}$ both in $E$ and in $X$,  where $\Q_T$  is given by
\begin{equation}
(\mathbb Q_T \,h)(t)=\int_{0}^{T} K(t,s)h(s)\, ds, \quad t \in [0,T],\;h\in X
\label{e13}
\end{equation}
where
\begin{equation}
\label{e14}
K(t,s)=\left\{\begin{array}{l}
\ds\int_{0}^se^{(t-r)A}Ce^{(s-
r)A^*}dr \quad\mbox{\rm if}\;0\le s\le t\le T\\
\\
\ds\int_{0}^{t}e^{(t-r)A}Ce^{(s-
r)A^*}dr\quad\mbox{\rm if}\;0\le t\le s\le T.
\end{array}\right.
\end{equation}
\end{Lemma}

We note that there exists    an orthonormal basis $(e_j)$ on $X$ and a sequence $(\lambda_j)$ of nonnegative numbers such that
$$
\Q_Te_j=\lambda_je_j,\quad \,j\in\N,
$$
and an integer $k_0\ge 0$ such that
$$
\lambda_1=\lambda_2=\cdots =\lambda_{k_0} =0,\quad 
\lambda_j>0,\quad\forall\,j>k_0.
$$
If  $k_0=0$ then $\Q_T $ is non degenerate.

 We  shall denote by   $ \mathbb  L_T$   the linear operator  from $X$ into itself defined by 
 \begin{equation}
 \label{e15}
\mathbb  L_Th(t)= \int_0^t e^{(t-s)A}\sqrt{C}\,h(s)ds,\quad h\in X,\;t\in[0,T].
 \end{equation}
 Its adjoint   $ \mathbb  L_T^{*}$  is given by
 \begin{equation}
 \label{e16}
\mathbb L_T^{*}g(t)=\int_t^T\sqrt Ce^{(s-t)A^*} g(s)ds,\quad g \in X,\;t\in[0,T].
 \end{equation}
 It is easily checked that
  $ \Q_T=\mathbb L_T\,\mathbb L_T^*$. Moreover,
 by \cite[Corollary B5]{DaZa14}, the Cameron--Martin space of the Gaussian measure $ N_{\mathbb Q_T}$
   is given by
 \begin{equation}
 \label{e18}
 \Q_T^{1/2}(X)=\mathbb L_T(X),
 \end{equation}
  both in $E$ and in $X$. 
  \begin{Remark}
  \label{r}
  \em If det $C=0$ one checks easily that  the gaussian measure $ \Q_T$ is degenerate and
  $$
  \mbox{\rm Ker} \,Q=\{h\in X:\,\mathbb L_T^*h=0\}.
  $$
  For instance, coming back to Example \ref{ex1}  we see that in that case
  $$
\mbox{\rm Ker}\;Q=\{h=(h_1,\,h_2)\in E:\, h'_1(t)=h_2(t)
\}.
$$
  \end{Remark}
  We shall denote
 in what follows  by $\mathbb Q_T^{-1}$ (resp.  $\mathbb Q_T^{-1/2}$) the pseudo--inverse of $\mathbb Q_T$ (resp.    the pseudo--inverse of $\mathbb Q_T^{1/2}$) \footnote{Let  $S:X\to Y$ be a linear, bounded and compact operator; the pseudo--inverse $S^{-1}$ of $S$ is defined as follows. For any $y\in S(X)$  we denote by $S^{-1}y$ the element of minimal norm from the convex set $\{x\in X:\,S(x)=y\}$.}.
 Clearly,  the domain of  $\mathbb Q_T^{-1}$ is equal to $\Q_T(X)$ and 
 $h\in  \mathbb Q_T^{-1}(X)$ if and only if the following series is convergent in $X$
 $$
 \mathbb Q_T^{-1}h=\sum_{j=k_0+1}^\infty \lambda^{-1}_j\langle h,e_j\rangle_X\,e_j.
 $$
 Similar assertion holds for $\Q_T^{-1/2}$.\medskip

Now to introduce   the required translation;
 first we need a lemma. 
\begin{Lemma}
\label{l5b}
Define $U:=  \int_{0}^Tr e^{rA}Ce^{rA^*}dr.$
 Then $\det U>0$.
 \end{Lemma}
 \begin{proof}
 We have in fact 
 $$
U\ge \frac{T}2\int_{T/2}^T e^{rA}Ce^{rA^*}dr=\frac{T}2
\int_{0}^{T/2} e^{(T/2+z)A}Ce^{(T/2+z)A^*}dz=\frac{T}2\,e^{AT/2}\,Q_{T/2}\,e^{A^*T/2}.
$$
It follows that $\det U\ge  \frac{T}2\,e^{T\,{\rm Tr}\,A}\,\det Q_{T/2}>0,$ as claimed.
 \end{proof}
 
 \begin{Proposition}
 \label{p5c}
 For all $x\in H$ set
\begin{equation}
\label{e18e}
u(x,t):=e^{(T-t)A^*}U^{-1} e^{TA}x,\quad t\in[0,T]
\end{equation}
and define $a(x,\cdot):=\Q_T u(x,\cdot)$.
Then  it results $a(x,T)=e^{TA}x.$
Moreover,  there is $c_T,\,c_{1,T}>0$ such that
   \begin{equation}
   \label{e31e}
   |u(t,x)|_H\le c_T
   |x|_H,\quad \forall\,t\in[0,T],\,x\in H,
   \end{equation}
   and
   \begin{equation}
   \label{e25}
   |a(x,t)|_H\le  c_{1,T} |x|_H,\quad \forall\,t\in[0,T],\,x\in H.
   \end{equation}
\end{Proposition}
\begin{proof}
Write 
$$\begin{array}{l}
\ds a(x,T)=\int_{0}^{T} K(T,s)u(x,s)\, ds=\ds \int_{0}^T\left(\int_{0}^se^{(T-r)A}Ce^{(s-
r)A^*}dr\right)e^{(T-s)A^*}U^{-1}e^{TA}xds\\
\\
=\ds \int_{0}^T\left(\int_{0}^{s}e^{(T-r)A}Ce^{(T-r)A^*}dr\right)U^{-1}e^{TA}xds
=\ds \int_{0}^T(T-r)e^{(T-r)A}Ce^{(T-r)A^*}dr\,U^{-1} e^{TA}x=e^{TA}x,
\end{array}
$$
as required.
 Finally,
 $$
 |u(x,t)|_H\le \sup_{s\in [0,T]}\,\|e^{sA}\|^2_{\mathcal L(H)}\|U^{-1}\|_{\mathcal L(H)}\, |x|_H,\quad t\in [0,T],
 $$
so that 
  \eqref{e31e} and  \eqref{e25} follow easily.
\end{proof}
Now we   prove the first new result of the paper.
\begin{Proposition}
\label{p1}
Under Hypotheses \ref{h1} and \ref{h2} the semigroup $R^{\mathcal O_r}_{T},\;T>0$,
is strong Feller.
\end{Proposition}
\begin{proof}
Let $\varphi\in B(\overline{\mathcal O_r})$ and $x_0,x\in \overline{\mathcal O_r}$. Then by \eqref{e2.8} we find
$$
\begin{array}{l}
\ds |R^\mathcal O_T\varphi(x)-R^\mathcal O_T\varphi(x_0)|
\le \ds\|\varphi\|_\infty\int_{X} \,\left|\exp\left\{-\tfrac12\,F(x)+G(x,h)\right\}-\exp\left\{-\tfrac12\,F(x_0)+G(x_0,h)\right\}\right|N_{\mathbb Q_T}(dh) \\
\\
 \ds+\|\varphi\|_\infty\int_{X} {\mathds 1}_{{\{\Gamma(h+ d(x_0,\cdot) \le r\}\setminus \{\Gamma(h+ d(x,\cdot) \le r\}}}\,\exp\left\{-\tfrac12\,F(x_0)+G(x_0,h)\right\}N_{\mathbb Q_T}(dh)=:A_1+A_2.
\end{array}
$$
Taking into account    \eqref{e31e} we have
$$F(x)=\langle \Q_T^{-1}a(x,\cdot) ,a(x,\cdot)  \rangle_X=
  \langle u(x,\cdot), \Q_Tu(x,\cdot)\rangle_H\le \|\Q_T\|_{\mathcal L(H)}c^2_T\,
   |x|^2_H$$
  and
  $$|G(x,h)|=| \langle \Q_T^{-1}a(x,\cdot) ,h  \rangle_X|
 \le c_TT|x|_H|h|_X$$
Therefore
$$
\begin{array}{l}
\ds\exp\left\{-\tfrac12\,F(x)+G(x,h)\right\}-\exp\left\{-\tfrac12\,F(x_0)+G(x_0,h)\right\}\\
\\
\ds=\int_0^1\exp\left\{-\tfrac12\,F((1-\alpha)x_0+\alpha x)+G((1-\alpha)x_0+\alpha x,h))\right\}(x-x_0)\,d\alpha\\
\\
\ds\le \int_0^1\exp\left\{G((1-\xi)x_0+\xi x,h))\right\}|x-x_0|\,d\alpha\le\exp\{c_TT|x|_H|h|_X\}|x-x_0|.
\end{array}
$$ 
It follows that
$$
A_1\le \|\varphi\|_\infty\int_X\exp\{c_TT|x|_H|h|_X\}\,dN_{\Q_T}\,|x-x_0|.
$$
Since the integral above is finite we have $\lim_{x\to x_0} A_1=0.$
Concerning $A_2$ we have $\lim_{x\to x_0} A_2=0$
by the continuity of $d(x,\cdot)$ and the dominated convergence theorem.
The proof is complete.
\end{proof}

 \section{Approximating semigroup}

 We   define an approximating semigroup  $R^{\mathcal O_r}_{T,n}\varphi$, $T>0,$  on 
 $B_b(\overline{\mathcal O_r})$ setting for all $n\in\N$,
   \begin{equation}
  \label{e32a}
R^{\mathcal O_r}_{T,n}\varphi(x)=\int_{\{\Gamma_{n}(h+ d(x,\cdot)) \le r\}} \varphi(h(T)) \exp\{-\tfrac12\,F(x)+G^n(x,h)\} N_{\mathbb Q_T}(dh),
 \end{equation}
 where $F(x),\;x\in \overline{\mathcal O_r}$ is defined by \eqref{e12h}, $d(x,t)$ by \eqref{e16h} and $\Gamma_n$ by
 \begin{equation}
  \label{e32f}
\Gamma_n(h+ d(x,\cdot))= \sup\{g(h(t_j)+ d(x,t_j)),\;t_j=\tfrac{jT}{2^n},\;j=0,1,...,2^n\},\quad \forall h\,\in E,\,n\in\N
 \end{equation}
  and
 \begin{equation}
  \label{e31c}
 G^n(x,h)= \sum_{j=1}^{2^n} (u(x,t_j)\cdot h(t_j))\,(t_j-t_{j-1}),\quad x\in \overline{\mathcal O_r},\;h\in E.
   \end{equation} 
   \begin{Lemma}
   \label{l1L}
 (i)  It results
   \begin{equation}
  \label{e29}
 |\Gamma_n(h+ d(x,\cdot))-\Gamma_n(h_1+ d(x,\cdot))|\le a+be^{|h|_E+|h_1|_E},\quad h_1,\,h_2\in E.
   \end{equation} 
  (ii)  Moreover $h\to\Gamma_n(h+ d(x,\cdot))$  belongs to $W^{1,2}(E,N_Q)$  
   $$
    |\Gamma'_n(h+ d(x,\cdot))\cdot d(x,\cdot)|
    \le(a+be^{2|h|_E}|d(x,\cdot)|.
   $$
 \end{Lemma}
 \begin{proof}
 (i) follows from Hypothesis \ref{h2}(ii)  
 and (ii) is a well known consequence of the local lipschitzianity of $\Gamma_n$.
 \end{proof}

    \begin{Proposition}
 \label{p0} 
 Under Hypotheses \ref{h1}, \ref{h2} for all   $\varphi\in B(\overline{\mathcal O_r})$ it results
 $$
 \lim_{n\to\infty}R^{\mathcal O_r}_{T,n}\varphi(x)=R^{\mathcal O_r}_{T}\varphi(x),\quad\,\forall\,\;x\in \overline{\mathcal O_r}.
 $$
\end{Proposition}
 \begin{proof} 
 Let $\varphi\in B_b(\overline{\mathcal O_r})$.
 Then
   \begin{equation}
  \label{e31}
  \begin{array}{lll}
 \ds |R^{\mathcal O_r}_{T,n}\varphi(x)-R^{\mathcal O_r}_T\varphi(x)|&\le&\ds \|\varphi\|_\infty \int_{\{\Gamma_n(h+ d(x,\cdot) \le r\}}  \exp\{-\tfrac12\,F(x)\}  \Big|\exp\{G^{\,n}(x,h)\}-\exp\{G(x,h)\}\Big| N_{\mathbb Q_T}(dh) \\
 \\
 &&\ds + \|\varphi\|_\infty\int_{\{\Gamma_n(h+ d(x,\cdot)) \le r\}\backslash \{\Gamma(h+ d(x,\cdot) \le r\}}\exp\{-\tfrac12\,F(x)+ G (x,h)\} N_{\mathbb Q_T}(dh). 
 \end{array}
 \end{equation} 
  Taking into account \eqref{e31e}, yields
  \begin{equation}
   \label{e31f}
   |G^n(x,h)|\le  \sum_{j=1}^{2^n} |u(x,t_j)|_H\, |h(t_j)|_H\,(t_j-t_{j-1})\le c_T|x|_{C(\overline{\mathcal O_r})}\,|h|_E,\quad x\in \overline{\mathcal O_r},\;h\in E
   \end{equation}
 Now, set 
$$  
   B_n:=\{ h:g(h(t_i)) +d(x,t_i)\le r,\; \;t_j=\tfrac{jT}{2^n},\;j=0,1,...,2^n\},\quad B
   =\{ h:\,g(h(t)) +d(x,t)\le r,\, t\in[0,T]\}.
   $$
   Then $ B\subset B_n$ and $\bigcap_{n\in \N}B_n=B$, so that $ N_\Q(B_n)\downarrow N_\Q(B)$ as $n\to\infty$.

 Moreover,
 \begin{equation}
   \label{e30}
 \lim_{n\to \infty}G^{\,n}(x,h)=G(x,h), \quad \forall\,h\in E,\, x\in \overline{\mathcal O_r}
  \end{equation}
   and   by \eqref{e18e} there is $c_T>0$ such that
    \begin{equation}
   \label{e31b}
   \exp\{G^{\,n}(x,h)\}\le e^{c_T |h|_E},\quad\forall\;h\in E,\,\forall\, x\in \overline{\mathcal O_r}.
    \end{equation}
The conclusion follows  from  the dominated convergence theorem.

 \end{proof}

 It useful to write  an expression of $R^{\mathcal O_r}_{T,n}\varphi$ as a finite dimensional integral. To this purpose we consider the linear mapping 
  \begin{equation}
  \label{e24}
E=C([0,T];H)\to H^{2^n},\quad h\to (h(t_1),h(t_2),\cdots,h(t_{2^n})),\quad t_j=\tfrac{jT}{2^n},\;j=0,1,...,2^n,
 \end{equation}
whose law   is obviously gaussian, say $ N_{\Q_{T,n}}.$ 
 Then for any $n\in\N$  and
   any
$\varphi:H^{2^n}\to\R$  bounded and Borel   we have,
\begin{equation}
\label{e41}
\begin{array}{lll}
\ds\int_E\varphi(h(t_1),h(t_2),\ldots,h(t_{2^n}))\, N_{\Q_T}(dh)
&= &\ds\E[\varphi(W_A(t_1),W_A(t_2),\ldots,W_A(t_{2^n}))]
\\ \\
&=&\ds \int_{H^{2^n}}\varphi(\xi_1,\ldots,\xi_{2^n})\, N_{\Q_{T,n}}(d\xi_1\cdots d\xi_{2^n}).
\end{array}
\end{equation}
 Now we can write the approximating semigroup as an integral over $H^{2^n}$, namely
  \begin{equation}
  \label{e32az}
R^{\mathcal O_r}_{T,n}\varphi(x)=\int_{\{\Gamma_n(\xi+ d(x,\cdot)) \le r\}} \varphi(\xi_{2^n}) \exp\{-\tfrac12\,F(x)+G^{\,n}(x,\xi)\} N_{\mathbb Q_{T,n}}(d\xi),\quad \xi\in H^{2^n},
 \end{equation}
 where
 $$
 \Gamma_n(\xi+ d(x,\cdot))=\sup\{ g(\xi_j+d(x,t_j)), \; t_j=\tfrac{jT}{2^n},\;j=0,1,\ldots,2^n\}
 $$
 and
 $$
 G^{\,n }(x,\xi)=\sum_{j=1}^{2^n} (u(x,t_j)\cdot \xi_j)\,(t_j-t_{j-1}),\quad x\in \overline{\mathcal O_r},\;\xi\in H^{2^n}.
$$
\begin{Proposition}
 \label{p5}
 $\Q_{T,n}$ has a bounded inverse for all $n\in\N$,  so the Cameron--Martin space of $ N_{\Q_{T,n}} $   is the whole $H^{2^n}$.
 \end{Proposition}
 \begin{proof} 
Let $n\in\N,$
then by \eqref{e41} we have
$$
\E[e^{i\lambda(\sum_{h=1}^{2^n}\langle \xi_h,W_A(t_h)\rangle_H}]=e^{-\frac12\lambda^2\langle  \Q_{T,n}\,\xi,\xi \rangle_{H^{2^n}}},\quad \xi=(\xi_1,\ldots,\xi_{2^n})
$$
We claim that if  $\Q_{T,n} \,\xi=0$ then $\xi=0$.
In fact, if $\Q_{T,n}\, \xi=0$, we have
$
\E[e^{i\lambda(\sum_{h=1}^{2^n}\langle \xi_h,W_A(t_h)\rangle_H}]=1
$
and so,
  \begin{equation}
 \label{e30k}
\sum_{h=1}^{2^n}\langle \xi_h,W_A(t_h)\rangle_H=0, \quad\,N_{ \Q_{n,T}}\mbox{\rm --a.s.}.
\end{equation}
Now, setting $L(t)=\int_0^te^{-sA}dW(s),$ 
and $\rho_i =e^{-t_iA^*}\xi_i,$ we have
$$
\langle \rho_1+\rho_2+\cdots+\rho_{2^n},L(t_1)\rangle+\langle \rho_2+\cdots +\rho_{2^n},L(t_2)-L(t_1)\rangle+\cdots +\langle \rho_{2^n},L(t_{2^n})-L(t_{2^n-1})\rangle=0.
$$
Multiplying both sides by $\langle \rho_1+\rho_2+\cdots+\rho_{2^n},L(t_1)\rangle$ and taking expectation, yields
$$ \rho_1+\rho_2+\cdots+\rho_{2^n}=0, $$
since $
\E(\langle v,L(t_1)\rangle^2)=\langle Q_{t_1}v,v\rangle$
and  $Q_{t_1}$ is non singular by Hypothesis \ref{h1}.

Similarly we obtain $\rho_k+\rho_{k+1}+\cdots+\rho_{2^n}=0$ for $k=2,3,\ldots,2^n$,
 which implies finally $\xi=0$.
   \end{proof}
 \section{Differentiating the approximating semigroup}

First note that by \eqref{e31c}
we have
\begin{equation}
  \label{e31h}
 (G^{\,n}_h(x,h)\cdot (d_x(x,\cdot)y)=
 \sum_{j=1}^{2^n} (u(x,t_j)\cdot d_x(x,t_j)y)\,(t_j-t_{j-1}),\quad x\in \overline{\mathcal O_r},\;h\in E. 
\end{equation} 
 \begin{Lemma}
 \label{l7}
 For all $x\in \overline{\mathcal O_r},\;y\in H$, $n\in\N$, $\varphi\in B_b(\overline{\mathcal O_r})$ we have 
 \begin{equation}
 \label{e40}
  D_xR^\mathcal O_{T,n}\varphi(x)\cdot y=:M_1(n,x,y)+M_2(n,x,y),
 \end{equation}
where 
 \begin{equation}
 \label{e41a}
 \begin{array}{lll}
\ds M_1(n,x,y)=\ds
 \int\limits_{\{\Gamma_n(h+d(x,\cdot))\le r\}}&& \varphi(h(T))\exp\left\{-\tfrac12\,F(x)+G^{\,n}(x,h)\right\}\\
\\
&&\ds\times\left(-\tfrac12\,F_x(x)y+G_x^{\,n}(x,h)y-G^{\,n}_h(x,h)\cdot (d_x(x,\cdot)y)\right)\,  N_{\Q_{T }}(dh)
 \end{array}
 \end{equation}
and
\begin{equation}
 \label{e42}
  M_2(n,x,y)=\hspace{-5mm}\ds\int\limits_{\{\Gamma_n(h+d(x,\cdot))\le r\}} \varphi(h(T))\exp\left\{-\tfrac12\,F(x)+G^{\,n}(x,h)\right\}\langle \Q_{T,n}^{-1/2} (d_x(x,\cdot )y), \Q_{T,n}^{-1/2}\,h\rangle_{H^{2^n}} \,  N_{\Q_{T}}(dh),
  \end{equation}
 where
   \begin{equation}
 \label{e41g}
   \langle \Q_{T,n}^{-1/2} (d_x(x,\cdot )y), \Q_{T,n}^{-1/2}\,h\rangle_{H^{2^n}}= \sum_{i,j=1}^{2^n} (\Q_{T,n}^{-1})_{i,j}(d_x(x,t_i)y)\cdot h(t_j),
 \end{equation}
 where $$(\Q_{T,n}^{-1})_{i,j}=\langle \Q_{T,n}^{-1}\psi_i,\psi_j\rangle_{H^{2^n}},\quad i,j=1,..,2^n$$
 and $(\psi_j)$ is the standard  orthogonal basis of $H^{2^n}$.
  \end{Lemma}
  \begin{proof}
  We first write identity \eqref{e32a} as
   $$
R^{\mathcal O_r}_{T,n}\varphi(x)=\int_{\{\Gamma_{n}(\xi+ d(x,\cdot)) \le r\}} \varphi(\xi_{2^n}) \exp\{-\tfrac12\,F(x)+G^n(x,\xi)\} N_{\mathbb Q_{T,n}}(d\xi),
$$
   Then we drop the dependence on $x$ under the domain of  integration by making the translation $\xi\to \xi- d(x,\cdot)$ 
    and recalling that $d(x,T)=0$; we write
\begin{equation}
 \label{e2.8c}
 R^{\mathcal O_r}_{T,n}\varphi(x)=\int_{\{\Gamma_n(\xi)\le r\}} \varphi(\xi_{2^n})\exp\left\{-\tfrac12\,F(x)+G^{\,n}(x,\xi- d(x,\cdot))\right\} N_{d(x,\cdot),\Q_{T,n}}(d\xi). 
  \end{equation}
 So,  using again the Cameron--Martin Theorem (this is possible thanks to Proposition \ref{p5}), we have
 \begin{equation}
 \label{e35a}
 R^{\mathcal O_r}_{T,n}\varphi(x)=\int_{\{\Gamma_n(\xi)\le r\}} \varphi(\xi_{2^n})\exp\left\{-\tfrac12\,F(x)+G^{\,n}(x,\xi- d(x,\cdot)\right\}\chi_n(x,\xi)\, N_{\Q_{T,n}}(d\xi),
 \end{equation}
where
\begin{equation}
 \label{e36a}
\chi_n(x,\xi)=\exp\left\{-\tfrac12| \Q_{T,n}^{-1/2} d(x,\cdot)|_{H^{2^n}}+\langle \Q_{T,n}^{-1/2}d(x,\cdot ),\Q_{T,n}^{-1/2}\xi \rangle_{H^{2^n}}\right\}.
  \end{equation}
We  now can  differentiate $R^{\mathcal O_r}_{T,n}\varphi(x)$  at any   given direction $y\in H$.
  Taking into account   that for any $x,y\in H$ we have,
 \begin{equation}
 \label{e37b}
D_x \chi_{n}(x,\xi)\cdot y=\langle \Q_{T,n}^{-1/2} (d_x(x,\cdot )y), \Q_{T,n}^{-1/2}\,(\xi -d(x,\cdot ))\rangle_{H^{2^n}} \,\chi_n(x,\xi).
 \end{equation}
 we find
 \begin{equation}
 \label{e38a}
 \begin{array}{lll}
\ds D_xR^{\mathcal O_r}_{T,n}\varphi(x)\cdot y &=&\ds\int_{\{\Gamma_n(\xi)\le r\}} \varphi(\xi_{2^n})\exp\left\{-\tfrac12\,F(x)+G^{\,n}(x,\xi- d(x,\cdot))\right\}\\\\
&&\ds\times \Big[-\tfrac12\,F_x(x)y+G_x^{\,n}(x,\xi- d(x,\cdot))y -\langle G^{\,n}_\xi(x,\xi- d(x,\cdot)),d_x(x,\cdot)y\rangle_{H^{2^n}}\\
\\
&&+\langle \Q_{T,n}^{-1/2} (d_x(x,\cdot )y), \Q_{T,n}^{-1/2}\,(\xi -d(x,\cdot ))\rangle_{H^{2^n}} \Big]\chi_n(x,\xi)\, N_{\Q_{T,n}}(d\xi).
 \end{array}
 \end{equation}
Here $F_x$ and $G_x$ denote the derivatives with respect to $x$  of $F$ and $G$ respectively, whereas  $G_\xi$ is
the derivative with respect to $\xi$.
  Now  making   the  opposite translation
 $\xi_j\to \xi_j+d(x,t_j),\, j=0,1,\ldots,2^n,$  we obtain  
\begin{equation}
 \label{e44}
 \begin{array}{l}
\ds D_xR^\mathcal O_{T,n}\varphi(x)\cdot y=
 \int_{\{\Gamma_n(\xi+d(x,\cdot))\le r\}} \varphi(\xi_{2^n})\exp\left\{-\tfrac12\,F(x)+G^{\,n}(x,\xi)\right\}\\
\\
\hspace{27mm}\ds\times\left(-\tfrac12\,F_x(x)y+G_x^{\,n}(x,\xi)y-\langle G^{\,n}_\xi(x,\xi),d_x(x,\cdot)y\rangle_{H^{2^n}}\right) \,  N_{\Q_{T,n}}(d\xi)\\
\\
 +\ds\int_{\{\Gamma_n(\xi+d(x,\cdot))\le r\}} \varphi(\xi_{2^n})\exp\left\{-\tfrac12\,F(x)+G^{\,n}(h,\xi )\right\}\langle \Q_{T,n}^{-1/2} (d_x(x,\cdot )y), \Q_{T,n}^{-1/2}\,(\xi)\rangle_{H^{2^n}} \,  N_{\Q_{T,n}}(d\xi).
 \end{array}
   \end{equation}
  We arrive at the conclusion making the change of variables \eqref{e24}.
  
 \end{proof}
 
 In Section 4 we shall easily prove the
 existence of the limit of $M_1(n,x,y)$ as $n\to\infty$. Instead a problem arises, as   said in the introduction,  for the term $M_2(n,x,y)$
due to the  factor
$$\langle \Q_{T,n}^{-1/2} (d_x(x,\cdot )y), \Q_{T,n}^{-1/2}\,h\rangle_{H^{2^n}},
$$
  because $d_x(x,\cdot )y$ does not  belong to $ \Q_{T}^{1/2}(X)$. So, we look in the next Lemma \ref{l12} for a different expression of $M_2(n,x,y)$ that does not contain this term.
  Before we need to recall  the definition  and some properties of the Sobolev space  $W^{1,p}(E,N_{\Q_T}).$ 
  We shall need a result   which is a straightforward generalisation of  
  \cite[Proposition 6.1.5]{Ce01}
 \begin{Lemma}
\label{l2z}
 For any $\varphi\in C^1_b(E)$ there exists a sequence $(\varphi_n)\in C^1_b(X)$ such that
\begin{itemize}
\item[(i)] $\ds\lim_{n\to \infty}\varphi_n(h)\to \varphi(h),\quad  \forall\;h\in E$.
\item[(ii)] $\ds 
\lim_{n\to \infty}\langle D\varphi_n(h),\eta\rangle_X=D\varphi(h)\cdot \eta,$ $\quad\forall\,h,\eta\in E$.
\end{itemize}
\end{Lemma}
 \begin{proof} 
For any $\varphi\in C^1_b(E)$   set
 $$
 \varphi_n:H\to E,\quad x\to  \varphi_n(x)(t)=\frac n2\int_{t-\frac 1n}^{t+\frac 1n}\hat{\varphi}(s)\,ds,\ \ \ \ t \in\,[0,T],
 $$
and $\hat{\varphi}(s)$ is the extension by oddness of $\varphi(s)$, for $s \in (-T,0)$ and $s \in\,(T,2T)$.
Then it is easy to check  that $(\varphi_n)$ fulfills (i) and (ii).
\end{proof}

The following result is similar to 
\cite[Proposition 4.2]{BoDaTu18}.
\begin{Proposition}
\label{p1z}
For all $\varphi\in C^1_b(E)$ and any $\eta\in Q^{1/2}(X)\subset E$ the following integration by parts formula holds
\begin{equation}
\label{e1w}
\int_E D\varphi(h)\cdot \eta   \,N_{\Q_T}(d h)=\int_E \varphi(h)\,\langle Q^{-1/2}h, Q^{-1/2}\eta  \rangle_H\,N_{\Q_T}(d h).
\end{equation}

\end{Proposition}
\begin{proof}
Let $\varphi_n\in C^1_b(X)$ be a sequence  
as in Lemma \ref{l2z}; then   we have
\begin{equation}
\label{e2w}
\int_H \langle D\varphi_n(h),\eta\rangle_X \,N_{\Q_T}(dh)
= \int_H \varphi_n(x)\,\langle Q^{-1/2}h, Q^{-1/2}\eta  \rangle_H\,dN_{\Q_T}(dh).
\end{equation}
The conclusion follows letting $n\to\infty.$
\end{proof}
\begin{Corollary}
For all $\varphi,\psi\in C^1_b(E)$ and any $\eta\in Q^{1/2}(X)\subset E$ the following integration by parts formula holds
\begin{equation}
\label{e3w}
\int_E D\varphi\cdot \eta \,\psi \, dN_{\Q_T}= - \int_E D\psi\cdot \eta \,\varphi \, dN_{\Q_T} +
\int_E \varphi\,\psi\,\langle Q^{-1/2}x, Q^{-1/2}z  \rangle_X \, dN_{\Q_T}.
\end{equation}
\end{Corollary}
 \begin{Remark}
 \label{rnew}
 \em
  By \eqref{e3w} it follows, by standard arguments, that the gradient operator $D$ is closable in $L^p(E,N_\Q)$ for any $p\ge 1$; we shall still denote by $D$ its closure and by $W^{1,p}(E,N_\Q)$ its domain. Finally, it is well known  that all Lipschitz continuous function  $\varphi:E\to \R$ belongs to $W^{1,p}(E,N_\Q)$.
  \textcolor{blue}{See also Lemma~\ref{l1L}}\end{Remark}

     Now we are  ready to prove the announced lemma.
    \begin{Lemma}
\label{l12}
 Assume Hypotheses \ref{h1} and \ref{h2}. Let $M_2(n,x,y)$ given by \eqref{e42}.
Then  for all $\varphi\in B_b(\overline{\mathcal O_r}),\,x\in\overline{\mathcal O_r},\,y\in H,\,n\in\N$, the following  identity holds
 \begin{equation}
\label{e45}
\begin{array}{lll}
  \ds M_2(n,x,y)&=&\ds \int_{\{\Gamma_n( h+d(x,\cdot)) \le r\}} \varphi(h(T))  D_h\exp\left\{-\tfrac12\,F(x)+G^{\,n}(x,h)\right\}\cdot (d_x(x,\cdot )y)\ N_{\Q_{T}}(dh) \\
 \\
&&+\ds\lim_{\epsilon\to 0}\frac1{2\epsilon}\int_{\{r-\epsilon\le\Gamma_n( h+d(x,\cdot)) \le r+\epsilon\}} \varphi(h(T))(\Gamma'_ n( h+d(x,\cdot))\cdot (d_x(x,\cdot )y) N_{\Q_{T}}(dh)\\
\\
&=&:M_{2,1}(n,x,y)+M_{2,2}(n,x,y).
  \end{array}
\end{equation}
 \end{Lemma}
\begin{proof}
By \eqref{e44} we have
$$
M_2(n,x,y)= \int_{\{\Gamma_n(\xi+d(x,\cdot))\le r\}} \varphi(\xi_{2^n})\exp\left\{-\tfrac12\,F(x)+G^{\,n}(h,\xi )\right\}\langle \Q_{T,n}^{-1/2} (d_x(x,\cdot )y), \Q_{T,n}^{-1/2}\,(\xi)\rangle_{H^{2^n}} \,  N_{\Q_{T,n}}(d\xi).
  $$
 Let us  first assume in addition that $\varphi\in C^1(\overline{\mathcal O_r})$.
Then we argue similarly  to   \cite[Proposition 4.5]{BoDaTu18}, defining a mapping $\theta_\epsilon:\R\to \R$,
\begin{equation}
\label{e1.4}
\begin{array}{l}
\theta_\epsilon(s)=\left\{\begin{array}{l}
0,\quad\mbox{\rm if}\;s\le r-\epsilon\\
\frac1{2\epsilon}(s-r+\epsilon),\quad\mbox{\rm if}\;r-\epsilon\le s\le r+\epsilon\\
1,\quad\mbox{\rm if}\;s\ge r+\epsilon.
\end{array}\right.
\end{array} 
\end{equation}
Then we  approximate $M_2(n,x,y)$ by setting
$$
  M^\epsilon_2(n,x,y)=\ds\int_{H^{2^n}} \theta_\epsilon(\Gamma_n( \xi+d(x,\cdot)))\varphi(\xi_{2^n})\exp\left\{-\tfrac12\,F(x)+G^{\,n}(x,\xi)\right\}\langle \Q_{T,n}^{-1/2} (d_x(x,\cdot )y), \Q_{T,n}^{-1/2}\,\xi\rangle_{H^{2^n}} \,  N_{\Q_{T,n}}(d\xi),
$$
so that the $\lim_{\epsilon\to 0}M^\epsilon_2(n,x,y)$ exists and is given by
  \begin{equation}
 \label{e42e}
  \lim_{\epsilon\to 0}M^\epsilon_2(n,x,y)=\ds\int_{\{\Gamma_n( \xi+d(x,\cdot))\le r\}} \varphi(\xi_{2^n})\exp\left\{-\tfrac12\,F(x)+G^{\,n}(x,\xi)\right\}\langle \Q_{T,n}^{-1/2} (d_x(x,\cdot )y), \Q_{T,n}^{-1/2}\,\xi\rangle_{H^{2^n}} \,  N_{\Q_{T,n}}(d\xi),
  \end{equation}
  Now, by a classical integration by parts formula we have, see e.g.
  \cite{Bo98}
 \begin{equation}
 \label{e55}
  \begin{array}{l}
  M^\epsilon_2(n,x,y)=\ds\int_{H^{2^n}} \theta_\epsilon(\Gamma_n( \xi+d(x,\cdot)))(D_h\varphi(\xi_{2^n})\cdot  d_x(x,T)y))\exp\left\{-\tfrac12\,F(x)+G^{\,n}(x,\xi)\right\}\  \,  N_{\Q_{T,n}}(d\xi)\\
  \\
  \ds+\int_{H^{2n}} \theta_\epsilon(\Gamma_n(\xi+d(x,\cdot)))\varphi(\xi_{2^n})\,(D_h\exp\left\{-\tfrac12\,F(x)+G^{\,n}(x,\xi)\right\}\cdot   d_x(x,\cdot )y))) \,  N_{\Q_{T,n}}(d\xi)\\
  \\
  \ds+\int_{H^{2^n}}(D_h \theta_\epsilon(\Gamma_n(\xi+d(x,\cdot))\cdot   (d_x(x,\cdot )y))\varphi(\xi_{2^n})\exp\left\{-\tfrac12\,F(x)+G^{\,n}(x,\xi)\right\}\ \,  N_{\Q_{T,n}}(d\xi).
  \end{array}
   \end{equation}
   Taking into account that the first integral vanishes, because $ d_x(x,T)y=0$, and that  
$$\begin{array}{lll}
\langle D\theta_\epsilon(\Gamma_na(\xi+d(x,\cdot))),(d(x,\cdot ))\rangle_{H^{2^n}}&=&\theta_\epsilon'(\Gamma_n(\xi+d(x,\cdot)))\,\langle  \Gamma'_n(\xi+d(x,\cdot)),d(x,\cdot )y\rangle_{H^{2^n}} \\
\\
&=&\frac1{2\epsilon}\,\langle  \Gamma'_n(\xi+d(x,\cdot)),d_x(x,\cdot )y\rangle_{H^{2^n}}{\mathds 1}_{[r-\epsilon.\,r+\epsilon]},
\end{array}$$
 we deduce  by \eqref{e55}, letting $\epsilon\to 0$, that 
  \begin{equation}
 \label{e56}
  \begin{array}{l}
  \ds M_2(n,x,y)= \int_{\{\Gamma_n(\xi+d(x,\cdot))\le r\}}  \varphi(\xi_{2^n})\,(D_h\exp\left\{-\tfrac12\,F(x)+G^{\,n}(x,\xi)\right\}\cdot   d_x(x,\cdot )y)) \,  N_{\Q_{T,n}}(d\xi)\\
  \\
  \ds+\lim_{\epsilon\to 0}\,\frac1{2\epsilon}\int_{\{r-\epsilon \le \Gamma_n(\xi+d(x,\cdot))\le r+\epsilon\}} \langle  \Gamma'_n(\xi+d(x,\cdot)),d(x,\cdot )y\rangle_{H^{2^n}}\varphi(\xi_{2^n})\exp\left\{-\tfrac12\,F(x)+G^{\,n}(x,\xi)\right\}\ \,  N_{\Q_{T,n}}(d\xi).
  \end{array}
   \end{equation}

 So the conclusion of the lemma follows, by the change of variables \eqref{e24},  when $\varphi \in C^1(\overline{\mathcal O_r}).$ The  case  when $\varphi\in  C(\overline{\mathcal O_r})$ can be handled by a uniform approximation of $\varphi$ by $C^1(\overline{\mathcal O_r})$ functions. Finally, if $\varphi\in  B_b(\overline{\mathcal O_r})$ we conclude using the strong Feller property of the semigroup, see Proposition \ref{p1}.
 \end{proof}
 
 It remains to compute the limit in \eqref{e45}. This we will do using the Ehrhard inequality.
 
\subsection{Applying the Ehrhard inequality}

Define
\begin{equation}
\label{e57}
\Lambda_{x}(s):= N_{\Q_{T}}(\Gamma(h+d(x,\cdot))\le s),\qquad 
 \Lambda_{n,x}(s):= N_{\Q_{T}}(\Gamma_n(\xi+d(x,\cdot))\le s),\quad \forall\,s>0, \,x\in \overline{\mathcal O_r},\,n\in\N.
\end{equation}
Since $g$ is convex by Hypothesis \ref{h2}(i),
 the mapping $\Gamma(\cdot+d(x,\cdot))$ (resp.
$\Gamma_{n}(\cdot+d(x,\cdot))$) is convex as well.
By applying the  Ehrhard  inequality (see e.g.    \cite[Th. 4.4.1]{Bo98}) we see that for any $x\in \overline{\mathcal O_r}$ the real function 
$$
[0,+\infty)\to\R,\; s\to S_{x}(s):=\Phi^{-1}(\Lambda_{x}(s)), (\mbox{\rm resp.}\, [0,+\infty)\to\R,\; s\to S_{n,x}(s):=\Phi^{-1}(\Lambda_{n,x}(s))),
$$
where $$\Phi(z)=\frac1{\sqrt{2\pi}}\int_{-\infty}^z e^{-\frac12 v^2}\,dv,\quad z\in \R,$$
is concave. Note that $\Phi^{-1}: (0,1)\to (-\infty,+\infty).$
As a consequence,  $\Lambda_{x}(\cdot)$ (resp. $\Lambda_{n,x}(\cdot)$)
is
  differentiable at any $s>0$ 
up to a discrete set  $N_s$ where there exist the left and the right derivative; we shall 
denote by  $D^+_r\Lambda_{x}(s)$  (resp $D^+_r\Lambda_{n,x}(s)$)  the right derivatives at any discontinuity point, and also (with the same symbol)   the derivative at the other points.

It follows  that $N_{\Q_{T}}\circ (\Gamma(h+d(x,\cdot)))^{-1}$ (resp.  $N_{\Q_{T}}\circ (\Gamma_n(h+d(x,\cdot)))^{-1}$) is absolutely continuous with respect to the Lebesgue measure  $\ell$
 and it results
\begin{equation}
\label{e601}
\frac{dN_{\Q_{T}}\circ (\Gamma(h+d(x,\cdot)))^{-1}}{d\ell}(s)=D^+_r\Lambda_{x}(s), \qquad (\mbox{\rm resp.}\;\frac{dN_{\Q_{T}}\circ (\Gamma_{n}(\xi+d(x,\cdot)))^{-1}}{d\ell}(s)=D^+_r\Lambda_{n,x}(s))
,\;s>0.
\end{equation}

Note  that for any $x$, $ \Lambda_{n,x}(s)$ is increasing on $s$ and decreasing on $n$. Moreover, $\Lambda_{n,x}(0)=0$ and $ \Lambda_{n,x}(s)\uparrow 1$ as $s\to\infty$.
Also \begin{equation}
\label{e601z}D^+S_x(s)=\sqrt{2\pi} \,e^{\frac12\,S^2(s)}\,D^+ \Lambda_x(s)\end{equation}\medskip

Now we are going  to   estimate  $D^+_r \Lambda_{n,x}(s)$ independently of $n$,  $x$ and $s\in [r/2,3r/2]$. Then we shall show that
$D^+_r \Lambda_{n,x}\to D^+_r \Lambda_{x}$.
\begin{Lemma}
\label{l13s}
There exists $K_r>0$ independent of $x,\,n,\,s$ such that
 \begin{equation}
 \label{e60}
 D^+_r \Lambda_{n,x}(s)\le K_r,\quad\forall\, x\in \overline{\mathcal O_r},\,\forall\,n\in\N,\, \forall\,s\in [r/2,3r/2].
 \end{equation}
 Moreover, it results
\begin{equation}
 \label{e60l}
 \lim_{n\to\infty}D^+_r \Lambda_{n,x}=D^+_r \Lambda_{x},\quad \forall\,x\in \overline{\mathcal O_r}.
 \end{equation}
\end{Lemma}

 \begin{proof}
 We proceed in three steps.\medskip
  
 {\bf Step 1}.  There is $l_1>0$ such that
 \begin{equation}
\label{e61}
0<l_1\le\Lambda_{n,x}(s),\quad \forall\;x\in \overline{\mathcal O_r},\,\forall\,n\ge 2, \forall\,s\in [r/2,3r/2].
\end{equation}

 It is enough to show \eqref{e61} for $\Lambda_{2,x}(s)$, because $\Lambda_{2,x}(s)\ge \Lambda_{n,x}(s)$  for $n\ge 2$.
 
Since  the convex set  $ \{\xi\in H^{2^n}:\,\Gamma_{2} (\xi+d(x,\cdot))< s\}$   is open and non empty  and the measure $N_{\Q_{T,2}}$ is non degenerate by Proposition \ref{p5}, it follows that there is $l_1(x)$ such that
$$
0<l_1(x)\le \Lambda_{n,x}(s),\quad \forall\;x\in \overline{\mathcal O_r},\,\forall\,n\ge 2, \forall\,s\in [r/2,3r/2].
$$
Since $\overline{\mathcal O_r}$ is compact, the conclusion follows.\medskip

 {\bf Step 2}.  There is $l_2<1$ such that
 \begin{equation}
\label{e62}
\Lambda_{n,x}(s)<l_2<1,\quad \forall\;x\in \overline{\mathcal O_r},\,\forall\,n\in\N,  \forall\,s\in [r/2,3r/2].
\end{equation}

In fact, thanks to Hypothesis \ref{h2}(ii),
there exists $M>0$ such that
\begin{equation}
\label{e63}
\Lambda_{n,x}(s)\le M,\quad \forall\;x\in \overline{\mathcal O_r},\,\forall\,n\in\N,  \forall\,s\in [r/2,3r/2].
\end{equation}

 {\bf Step 3}.  Conclusion.\medskip

 Note first that
 $$
 S_{n,x}(s)\downarrow S_{x}(s)\quad \mbox{\rm as}\;n\to\infty,\; x\in \overline{\mathcal O_r}. 
$$
The sequence  $(S_{n,x}(\cdot))$ is obviously increasing   
and  also  concave   by the Ehrhard inequality. Therefore,  all elements of  $(S'_{n,x}(\cdot))$ are positive and decreasing; so, they are BV  in the interval  $[r/2,3r/2].$

We claim  that the sequence  $(S'_{n,x}(r))$    is equi-bounded in $[r/2,3r/2]$  in $BV$ norm.
To show this fact it is enough to see that $(S'_{n,x}(r))$ is namely equi-bounded at $r_1$ (because  it is  decreasing). In fact, since $S_{n,x}$ is  concave we have if $0<\epsilon\le \tfrac{r}2$,
\begin{equation}
\label{e4.4a}
S'_{n,x}(r_1)\le \tfrac1\epsilon(S_{n,x}(r_1+\epsilon)-S_{n,x}(r_1)\le \tfrac2\epsilon\,S_{n,x}(r_2)= \tfrac2\epsilon\,\Phi^{-1}(r_2).
\end{equation} 
 Therefore we can apply the selection principle of Helly, see e.g. \cite[Theorem 5 page 372]{KoFo70} to the sequence   $(S'_{n,x}(\cdot))$ and conclude that there exists a subsequence of $(S'_{n,x}(\cdot))$ still   denoted by   $(S'_{n,x}(\cdot))$  that converges in all points of $[r/2,3r/2]$ to a function   $f(x,\cdot)$. 
 
 We claim that $f(x,s)$ is the derivative  of $S_{x}(s)$ in 
$[r/2,3r/2].$  This follows by an elementary
 argument writing,
  $$
 S_{n,x}(s)=\int_{r_1}^{s} S'_{n,x}(v)dv ,\quad s\in [r/2,3r/2].
  $$
  (recall that $S_{n,x}(\cdot)$ is absolutely continuous by \cite[Corollary  4.4.2]{Bo98}).
By the Dominated Convergence theorem it follows that for 
  $k\to \infty$ we have
  $$
  S_{x}(s)=\int_{r_1}^sS'_{x}(v)dv ,\quad s\in [r/2,3r/2],
  $$
 which implies
 $$
  S'_{x}(s)=f(x,s) ,\quad s\in[r/2,3r/2],
  $$
  as required.
  
 Therefore there is  a subsequence of   $(S'_{n,x})$  which converges to  $S'_x$ and consequently all the sequence   $(S'_{n,x})$ will converges to  $S'_x$ Thus  $\Lambda_x(r)$  has the right derivative for  $r\in [r/2,3r/2]$, $x\in 
\overline{\mathcal O_r},$  and \eqref{e60l} follows. 
  
  Finally, taking into account  \eqref{e4.4a} and \eqref{e601z},  it results
$$
D^+_r \Lambda_x(r)=\frac1{\sqrt{2\pi}}\,e^{-\frac12\,S_x^2(r)}\,S'_x(r)=\le   \tfrac1{\pi\epsilon}\,\Phi^{-1}(r_2).$$
So, \eqref{e60} follows.
\end{proof}

   The next lemma is devoted to the  computation   of  $\lim_{\epsilon\to 0}\,M_{2,2}(n,x,y)$, defined by \eqref{e45}.
\begin{Lemma}
\label{l16b}
Let   $n\in\N,$   $r>0$,
  $x\in\overline{\mathcal O_r}, y\in H$. Then it results
 \begin{equation}
\label{e64}
 M_{2,2}(n,x,y)=\E_{N_{\Q_{T}}}[\varphi(h(T))\,(\Gamma'_{n}(h+d(x,\cdot))\cdot(d_x(x,\cdot )y)|\Gamma_{n}(h+d(x,\cdot))=r]\,D^+_{n,r}\Lambda_{x}(r).
\end{equation}
  \end{Lemma}
  
\begin{proof}

Let us recall  that by \eqref{e45} we have,  
 $$
 \begin{array}{l}
 \ds M_{2,2}(n,x,y)=\lim_{\epsilon\to 0}\frac1{2\epsilon}\int_{\{r-\epsilon\le\Gamma_{n}(h+d(x,\cdot)) \le r+\epsilon\}} \varphi(h(T))(\Gamma'_{n}(h+d(x,\cdot))\cdot (d_x(x,\cdot )y) N_{\Q_{T}}(dh),
 \end{array} 
 $$
Taking into account \eqref{e601} it follows that   
   $$
  \begin{array}{l}
 \ds M_{2,2}(n,x,y)=\lim_{\epsilon\to 0}\frac1{2\epsilon}\,\int_{r-\epsilon}^{r+\epsilon}\E_{N_{\Q_{T}}}[  \varphi(h(T))(\Gamma'_n(h+d(x,\cdot))\cdot (d_x(x,\cdot )y)|\Gamma_n(h+d(x,\cdot))=s]\,D_{n,r}\Lambda_{x}(s)ds.
 \end{array}
 $$ 
 Note that the existence of a regular distribution of
 $$
 \E_{N_{\Q_{T}}}[  \varphi(h(T))(\Gamma'_n(h+d(x,\cdot))\cdot (d_x(x,\cdot )y)|\Gamma_n(h+d(x,\cdot))=s]
 $$
 is granted because $E$ is separable, see \cite[10.2.2]{Du02}.   
 
 It follows that
 $$
  \begin{array}{l}
 \ds M_{2,2}(n,x,y)=\E_{N_{\Q_{T}}}  \varphi(h(T))(\Gamma'_{n}(h+d(x,\cdot))\cdot (d_x(x,\cdot )y)|\Gamma_{n}(h+d(x,\cdot))=r]\,D^+_{n,r}\Lambda_{x}(r),
 \end{array}
 $$ 
 by virtue of the dominated convergence theorem.
 \end{proof}

  We prove now the following result.
  \begin{Proposition}
  \label{p13}
  Assume Hypotheses \ref{h1} and \ref{h2}  and let $n\in\N$. Then we have
 \begin{equation}
 \label{e68h}
 \begin{array}{lll}
D_xR^\mathcal O_{T,n}\varphi(x)\cdot y&=&\ds
\hspace{-5mm} \int\limits_{\{\Gamma_n(h+d(x,\cdot))\le r\}}\hspace{-5mm}\varphi(h(T))\exp\left\{-\tfrac12\,F(x)+G^{n}(x,h)\right\} \left(-\tfrac12\,F_x(x)y+G_x^{n}(x,h)y\right) \,  N_{\Q_{T }}(dh)\\
\\
 &&+\ds
 \E_{N_{\Q_T}}[\varphi(h(T))\,(\Gamma'_{n}(h+d(x,\cdot))\cdot(d_x(x,\cdot )y)|\Gamma_{n}( h+d(x,\cdot))=r]\,D^+_{n,r}\Lambda_{x}(r).
 \end{array}
 \end{equation}
 Moreover, there is $c_{2,T}(r)>0$ such that the following estimate holds
  \begin{equation}
 \label{e68i}
 \begin{array}{l}
\ds |D_xR^\mathcal O_{T,n}\varphi(x)|\le \|\varphi\|_\infty \,c_{2,T}(r)+\|\varphi\|_\infty\,\exp\left\{-\tfrac12\,F(x)\right\} 
\\
\\
\ds\hspace{20mm}\times
 \int_{\{\Gamma_{n}(h+d(x,\cdot))\le r\}} \exp\left\{c_T\,|h|_E\right\} \left(\tfrac12\,\|F_x(x)\|_{\mathcal L(H)}+Tc_{2,T}(r)  \,\|U^{-1}\|_{\mathcal L(H)}\,\|h\|_E\right) \,  N_{\Q_{T }}(dh).\\
\\
 \end{array}
 \end{equation}
 \end{Proposition}
 \begin{proof}

 From  Lemmas \ref{l7}, \ref{l12} and \ref{l16b} we obtain
 $$
 D_xR^\mathcal O_{T,n}\varphi(x)\cdot y=M_1(n,x,y)+  M_{2,1}(n,x,y)+M_{2,2}(n,x,y)
 $$
 and so,
 $$
 \begin{array}{l}
\ds D_xR^\mathcal O_{T,n}\varphi(x)\cdot y=
 \int_{\{\Gamma_n(h+d(x,\cdot))\le r\}} \varphi(h(T))\exp\left\{-\tfrac12\,F(x)+G^{\,n}(x,h)\right\}\\
\\
\ds\times\left(-\tfrac12\,F_x(x)y+G_x^{\,n}(x,h)y-(G^{\,n}_h(x,h)\cdot (d_x(x,\cdot)y) \right) \,  N_{\Q_{T }}(dh)\\
\\
+\ds \int_{\{\Gamma_n( h+d(x,\cdot)) \le r\}} \varphi(h(T))(D_h\exp\left\{-\tfrac12\,F(x)+G^{\,n}(x,h)\right\}\cdot (d_x(x,\cdot )y))  N_{\Q_{T}}(dh) \\
 \\
 +\ds
 \E[\varphi(h(T))\,(\Gamma'_{n}(h+d(x,\cdot))\cdot(d_x(x,\cdot )y)|\Gamma_n( h+d(x,\cdot))=r]\,D^+_r\Lambda_{n,x}(r).
 \end{array}
$$
 Since
 $$
 D_h\exp\left\{G^{\,n}(x,h)\right\}\cdot (d_x(x,\cdot )y) =\exp\{G^{\,n}(x,h)\}\, G_h^{\,n}(d_x(x,\cdot )y\cdot (d_x(x,\cdot )y), 
 $$
 we obtain letting $n\to\infty$, after some simplifications,  identity \eqref{e68h}.    
 Finally, we prove \eqref{e68i}. First by \eqref{e31b}  we have
  $$
 \exp\left\{G^{\,n}(x,h)\right\}\le  \exp\left\{c_T\,|h|_E\right\}. 
$$
 Moreover by \eqref{e31c}
  it follows that 
  $$G_x^{\,n}(x,h)= \sum_{j=1}^{2^n} u_x(x,t_j)\cdot h(t_j)\,(t_j-t_{j-1}),\quad x\in \overline{\mathcal O_r},\;h\in E$$
and  therefore we have
   \begin{equation}
 \label{e68l}
|G_x^{\,n}(x,h)|\le   T\|u_x(x,\cdot\|_{\mathcal L(H)}\,\|h\|_E\le Tc^2_T\,\|U^{-1}\|_{\mathcal L(H)}\,\|h\|_E,\quad x\in \overline{\mathcal O_r},\;h\in E
 \end{equation}
 Finally, by Hypothesis \ref{h2}(ii)  there exists $c_{1_T}>0$
 such that
 $$
| \Gamma'_{n}(h+d(x,\cdot))\cdot(d_x(x,\cdot )y)|
\le c_{1,T}\,|h|_E.
 $$
 Finally,  taking into account \eqref{e60} and  Lemma \ref{l1L}, yields
   \begin{equation}
 \label{e68m}
 \begin{array}{l}
\ds \left|
\E_{N_{\Q_T}}[\varphi(h(T))\,(\Gamma'_{n}(h+d(x,\cdot))\cdot(d_x(x,\cdot )y)|\Gamma_{n}( h+d(x,\cdot))=r]\,D^+_r\Lambda_{n,x}(r)\right|\\
\\
\hspace{10mm}\ds \le \|\varphi\|_\infty c_{1,T}\,|D^+_r\Lambda_{x}(r)|\le  \|\varphi\|_\infty c_{1,T}\,K_r.
 \end{array}
  \end{equation}
  The conclusion follows.
   \end{proof}

 \section{Main results}

Now we take $\varphi\in B_b(H)$ and   prove a representation formula for $D_xR^{\mathcal O_r}_{T}\varphi(x)$.
\begin{Theorem}
\label{t15}
Assume Hypotheses \ref{h1} and \ref{h2}. Then there exists the gradient
of $R^{\mathcal O_r}_T\varphi$ for all $\varphi\in B_b(\overline{\mathcal O_r})$ and it results
 \begin{equation}
 \label{e100}
 \begin{array}{l}
 D_xR^{\mathcal O_r}_{T}\varphi(x)\cdot y=\ds
 \int_{\{\Gamma(h+d(x,\cdot))\le r\}} \varphi(h(T))\exp\left\{-\tfrac12\,F(x)+G(x,h)\right\} \left(-\tfrac12\,F_x(x)y+G_x(x,h)y\right) \,  N_{\Q_{T }}(dh)\\
\\
\hspace{25mm} +\ds
\E_{N_{\Q_T}}[\varphi(h(T))\,(\Gamma'(h+d(x,\cdot))\cdot(d_x(x,\cdot )y)|\Gamma( h+d(x,\cdot))=r]\,D^+_r\Lambda_{x}(r).
 \end{array}
 \end{equation}
\end{Theorem}
\begin{proof}

We recall that by Proposition \ref{p13} we have
$$
D_xR^{\mathcal O_r}_{T,n}\varphi(x)\cdot y
= I(n,x,y)+J(n,x,y),
$$
where
  \begin{equation}
 \label{e65b}
I(n,x,y)= \int_{\{\Gamma_{n}(h+d(x,\cdot))\le r\}} \varphi(h(T))\exp\left\{-\tfrac12\,F(x)+G^{\,n}(x,h)\right\} \left(-\tfrac12\,F_x(x)y+G_x^{\,n}(x,h)y\right) \,  N_{\Q_{T }}(dh)
 \end{equation}
 and
  \begin{equation}
 \label{e66}
J(n,x,y)=
\E_{N_{\Q_T}}[\varphi(h(T))\,\Gamma_{n}(h+d(x,\cdot))\cdot(d_x(x,\cdot )y)|\Gamma_{n}( h+d(x,\cdot))=r]\,D^+_r\Lambda_{n,x}(r).
 \end{equation}
 {\bf Step 1}. Convergence of  $I(n,x,y)$ as $n\to\infty$.

 For all $x\in \overline{\mathcal O_r}$ and all  $y\in H$ we have
 \begin{equation}
 \label{e68a}
\lim_{n\to \infty}I(n,x,y)= \int_{\{\Gamma(h+d(x,\cdot))\le r\}} \varphi(h(T))\exp\left\{-\tfrac12\,F(x)+G(x,h)\right\} \left(-\tfrac12\,F_x(x)y+G_x(x,h)y\right) \,  N_{\Q_{T }}(dh)
 \end{equation}\medskip
 
 This follows   by the dominated convergence theorem arguing as in the proof of Proposition \ref{p0}.\bigskip   
 
 \noindent{\bf Step 2}.  Convergence of  $J(n,x,y)$ as $n\to\infty$. Let  $r>0$, $x\in\overline{\mathcal O_r}, y\in H$. Then,   we have, 
\begin{equation}
\label{e69a}
\lim_{n\to \infty}\,J(n,x,y)=-
\E[\varphi(h(T))\,(\Gamma'(h+d(x,\cdot))\cdot(d_x(x,\cdot )y)|\Gamma(h+d(x,\cdot))=r]\,D_r\Lambda_x(r).
\end{equation}
First we notice
  that $\Gamma_n(h+d(x,\cdot))$ converges uniformly to $\Gamma(h+d(x,\cdot))$ for any $x$.
Moreover, since  the function $h\to \sup_{t_j} h(t_j)$ is Lipschitz continuous in $E$ and $g$   fulfills Hypothesis \ref{h2}(ii), it follows that $\Gamma_n(h+d(x,\cdot))$  belongs to a bounded subset of $W^{1,2}(E,N_{\Q_T})$, by  Lemma \ref{l1L}.
So, a subsequence of $(\Gamma'_n(h+d(x,\cdot))$ (which we still denote by $(\Gamma'_n(h+d(x,\cdot))$) converges to $\Gamma'(h+d(x,\cdot))$ in $L^1(E,N_{\Q_{T}})$

Now we start from \eqref{e66} which we write as
 $$
  \begin{array}{l}
 \ds J(n,x,y)= \E[ \Psi_n(h)|\Gamma_{n}( h+d(x,\cdot))=r]\,D_r\Lambda_{n,x}(r),
 \end{array}
 $$
where 
 \begin{equation}
\Psi_n(h)=-
\varphi(h(T))(\Gamma'_{n}( h+d(x,\cdot))\cdot (d_x(x,\cdot )y).
\end{equation}
  Note that $D_r\Lambda_{n,x}(r)\to D_r\Lambda_{x}(r)$ as $n\to\infty$ by
 Lemma \ref{l13s}.
 
By Hypothesis \ref{h2}(ii) we have
$$|\Psi_n(h)|\le
 \|\varphi\|_\infty (a+e^{b|x|_{E}}),\quad \forall\,h\in E, $$
so that, there exists $M>0$ such that
$ |\Psi_n(h)|_{L^1(E,N_{\Q})} \le M,\, \forall\,n\in\N.$
 Also
$\Psi_n(h)\to \Psi(h)$ for  all $h\in E$ by Lemma \ref{l12}(iii).
Therefore $\Psi_n\to \Psi$ in $L^1(E,N_{\Q})$ by the dominated convergence theorem.

Now we can   show that
\begin{equation}
\label{e65}
 \lim_{n\to \infty}\E[\Psi_n|\Gamma_{n}(h+d(x,\cdot))=r]=\E[\Psi|\Gamma(h+d(x,\cdot))=r].
\end{equation}
To this aim write
$$
\begin{array}{l}
\big|\E[\Psi_n|\Gamma_{n}(h+d(x,\cdot))=r]-\E[\Psi|\Gamma(h+d(x,\cdot))=r]\big|\\
\\
\le\big|\E[\Psi_n-\Psi|\Gamma_{n}(h+d(x,\cdot))=r]\big| +\big|\E[\Psi|\Gamma_{n}(h+d(x,\cdot))=r]-\E[\Psi|\Gamma(h+d(x,\cdot))=r]\big|\\
\\
:=J_1(n)+J_2(n).
\end{array}
$$
Since $\Psi_n\to \Psi$ in $L^1(E,N_{\Q})$
we have
\begin{equation}
\label{e65e}
|J_1(n)|\to 0\quad \mbox{\rm in}\, L^1(E,N_{\Q})\qquad \mbox{as $n\to\infty$.}
\end{equation}
For dealing with $J_2(n)$, note that 
\begin{equation}
\label{e65ee}
\lim_{n\to\infty}\E[\Psi|\Gamma_{n}(h+d(x,\cdot))=r]=\E[\Psi|\Gamma(h+d(x,\cdot))=r]
\end{equation}
because $\Gamma_{n}$ is  decreasing to $\Gamma$,
see e.g. \cite[10.1.7]{Du02}.
Now  step 2    follows from \eqref{e65e}
and \eqref{e65ee}.

\medskip
 
 \noindent{\bf Step 3}. Esistence of $D_xR^\mathcal {O_r}_{T}\varphi$, for all  $\varphi\in C_b(\overline{\mathcal O_r})$.
 
 \medskip

\noindent Let us recall that by   Proposition \ref{p0} and  Steps 1,2  we know   that\medskip

(i)\; there exists the limit
$$\lim_{n\to \infty}R^{\mathcal O_r}_{T,n}\varphi(x)=R^{\mathcal O_r}_{T}\varphi(x),\quad\,\forall\,\;x\in \overline{\mathcal O_r} .$$

 (ii)\; there exists the limit
 $$
\lim_{n\to \infty}D_xR^{\mathcal O_r}_{T,n}\varphi(x)\cdot y=:\Xi(x)\cdot y \quad\forall\,x\in \overline{\mathcal O_r},\; y\in H.
$$
\indent (iii) There exists $M_{\|\varphi\|_\infty}>0$ such that
\begin{equation}
\label{e1000}
|R^{\mathcal O_r}_{T,n}\varphi(x)|+|D_xR^{\mathcal O_r}_{T,n}\varphi(x)|\le M_{\|\varphi\|_\infty},\quad\forall\,x\in \overline{\mathcal O_r}.
\end{equation}
Let now $x,x_0\in \overline{\mathcal O_r}.$ Since
$$
R^{\mathcal O_r}_{T,n}\varphi(x)-R^{\mathcal O_r}_{T,n}\varphi(x_0)=\int_0^1(D_xR^{\mathcal O_r}_{T,n}\varphi)(\alpha x+(1-\alpha)x_0)\cdot (x-x_0)\,d\alpha.
$$
 Letting $n\to \infty$ we obtain, by the dominated convergence theorem
 $$
R^{\mathcal O_r}_{T}\varphi(x)-R^{\mathcal O_r}_{T}\varphi(x_0)=\int_0^1\Xi(\alpha x+(1-\alpha)x_0)\cdot (x-x_0)\,d\alpha.
$$
This implies that $R^{\mathcal O_r}_{T}\varphi(x)$ is differentiable at $x$ in the direction $y$ and
$$
DR^{\mathcal O_r}_{T}\varphi(x)\cdot y=\Psi(x)\cdot y.
$$\medskip

\noindent\ {\bf Step 4}.$\varphi\in B_b(\overline{\mathcal O_r})$\bigskip

\noindent Since $R^\mathcal O_{T}$ is strong Feller (Proposition  \ref{p1}), we have
$R^\mathcal O_{T/2}\in C_b(H)$, so, the conclusion follows starting from $T/2$.

The proof is complete.
 
 \end{proof}\medskip
 


\begin{Example}
\label{ex16}
\em
 We continue here Example \ref{ex1}
Let $
A=\left(
\begin{array}{cc}
0&0\\
1&0
\end{array} \right),\;
C=\left(
\begin{array}{cc}
1&0\\
0&0
\end{array} \right).
$
Then we have $
e^{tA}=\left(
\begin{array}{cc}
1&0\\
t&1
\end{array} \right),\quad  e^{tA^*}=\left(
\begin{array}{cc}
1&t\\
0&1
\end{array} \right)$
We  have seen that Hypothesis \ref{h1} is fulfilled.
Let $U$ as in Lemma \ref{l5b}
\begin{equation}
U=  \int_{0}^Tr e^{rA}Ce^{rA^*}dr=\int_0^T\left(
\begin{array}{cc}
r&r^2\\
r^2&r^3
\end{array} \right)ds=\frac1{12}
\left(\begin{array}{cc}
6T^2&4T^3\\
4T^3&3T^4
\end{array}\right)
\end{equation}
so that det $U>0$ and
$$
 U^{-1}=\frac6{T^2}\,\left(\begin{array}{cc}
3&-\frac{4}{T}\\
-\frac{4}{T}&\frac{6}{T^2}
\end{array}\right)
 $$
  Moreover, by Proposition \ref{p5c} we have
$$
u(x,s)= \frac6{T^4}\,\begin{pmatrix}
T^2-2Ts&& 2(T-3s)\\
2T &&6
\end{pmatrix}\begin{pmatrix} x_1\\x_2\end{pmatrix}
,\quad x\in H,\;s\in[0,T].
$$
So,
\begin{equation}
\label{e77}
|u(x,s)|_H\le T^{-4}C_1(T)|x|_H
\end{equation}
and
\begin{equation}
\label{e78}
|K(t,s)|_H\le C_2(T),\quad \forall\,t,s\in[0,T],
\end{equation}
where $C_1(T), C_2(T)$ are continuous in $(0,+\infty)$. Moreover
\begin{equation}
\label{e79}
|a(t,x)|_H=|\Q_T\,u(t,x)|_H\le T^{-3}C_1(T)|x|_H
\end{equation}
Finally, $G(x,h)=\langle u(t,x),h\rangle_H\le |h|_X\, |u(t,x)|$.

Concerning Hypothesis \ref{h2}, assume that $g(x)=|x|^2$. Then
$$
\Lambda(x,r)=\int_{\|h+d(x,\cdot)\|_E\le r}N_{\Q_T}(dh)
$$
Note that
by  \eqref{e60}  we know that  $D_r\Lambda(x,r)$ is uniformly bounded in $x$.
\end{Example}


\newpage







\footnotesize
  
 \begin{thebibliography}{99}
 
 \bibitem[AiMa88]{AiMa88}
   H. Airault and  P. Malliavin.  \textit{Int\'egration g\'eom\'etrique sur l'espace de Wiener},  Bull. Sci. Math. {\bf 112},  3--52, 1988.
   
   
  \bibitem[Bo98]{Bo98}  V.I. Bogachev,  {\it Gaussian Measures}, American Mathematical Society, Providence, 1998.
  
  \bibitem[Ce01]{Ce01}  S. Cerrai, {\it Second order PDE's in finite and infinite dimension. A probabilistic approach}. Lecture Notes in Mathematics, 1762. Springer-Verlag, Berlin, 2001.
   
   \bibitem[BoDaTu18]{BoDaTu18}  S. Bonaccorsi, G. Da Prato and L. Tubaro, {\it Construction of a surface integral under local Malliavin assumptions, and related integration by parts formulas.} J. Evol. Equ. , no. 2, 871--897,  2018.

    

\bibitem[DaLuTu14]{DaLuTu14} G. Da Prato, A. Lunardi and  L. Tubaro,  {\it Surface measures in infinite dimension},   Rend. Lincei Mat. Appl. {\bf 25},   309--330, 2014.


 
  




 


   
  
 \bibitem[DaZa14]{DaZa14}  G. Da Prato and J. Zabczyk,  \textit{Stochastic equations in infinite 
dimensions,}
 Encyclopedia of Mathematics and its Applications,  Cambridge University
Press,  second edition, 2014.



 

\bibitem[Du02]{Du02}  R. M.  Dudley, {\it Real Analysis and Probability}, Cambridge studies in advanced mathematics, 2002.




\bibitem[KoFo70]{KoFo70}   A.N. Kolmogorov and S.V. Fomin, {\it Introductory real analysis}. Dover, New
York, 1970.

\bibitem[Li86]{Li86} W. Linde, {\it Gaussian measure of translated nballs in Banach spaces}, Theory Probab. Appl. {\bf 34}, no.2, 1986.

\bibitem[Nu06] {Nu06} D. Nualart,  {\it The Malliavin calculus and related topics}. Probability and its Applications, Springer-Verlag, 1995. Second Edition, Springer-Verlag, 2006.


\bibitem [Ph78]{Ph78} R. R. Phelps,   Gaussian null sets and differentiability of
Lipschitz map on Banach spaces, {\it Pac. J. Math.}, {\bf 77}, 523-531, 1978.


\bibitem[Ta00]{Ta00}  A. Talarczyk,   {\it Dirichlet problem for parabolic equations on
Hilbert spaces},   Studia Math.,  {\bf 141}, 109-142, 2000.


 

\bibitem[Za92]{Za92} J.Zabczyk, {\it Mathematical Control Theory: An Introduction},
Birkh\"auser, 1992.




\end{thebibliography}
\end {document}